\documentclass[12pt]{amsart}

\usepackage{amsfonts, amsthm, amsmath, amssymb, url, xspace, enumerate}
\usepackage{graphics, graphicx, epsf, color}

\usepackage{lineno}
\usepackage{fullpage}

\newtheorem{thm}{Theorem}[section]

\newtheorem{lem}[thm]{Lemma}

\newtheorem{cor}[thm]{Corollary}

\newtheorem{fact}[thm]{Fact}

\theoremstyle{definition}

\theoremstyle{remark}
\newtheorem{question}[thm]{Question}

\newcommand{\comment}[1]{}
\def\ignore#1{{ }}

\newcommand{\T}{\mathcal{T}}
\newcommand{\const}{\alpha}

\begin{document}


\author{H.A. Kierstead}
\address{Department of Mathematics and Statistics, Arizona State University,
Tempe, AZ 85287, USA. }
\email{kierstead@asu.edu}

\author{A.V. Kostochka}
\address{Department of Mathematics \\ University of Illinois \\ Urbana, IL 61801, USA\\ and Sobolev Institute of Mathematics\\ Novosibirsk, Russia}
\thanks{Research of this author is supported in part by NSF grant  DMS-1600592 and  by grants 15-01-05867  and 16-01-00499
 of the Russian Foundation for Basic Research. }
\email{kostochk@math.uiuc.edu}

\author{A. M\lowercase{c}Convey}
\address{Department of Mathematics \\ University of Illinois \\ Urbana, IL 61801\\ USA}
\thanks{This author gratefully acknowledges support from the Campus Research Board, University of Illinois.}
\email[Corresponding author]{mcconve2@illinois.edu}

\title{A sharp Dirac-Erd\H{o}s type bound for large graphs}

\date{\today}

\begin{abstract} 
Let $k \geq 3$ be an integer, $h_{k}(G)$ be the number of vertices of degree at least $2k$ in a graph $G$, and $\ell_{k}(G)$ be the number of vertices of degree at most $2k-2$ in $G$.
 Dirac and Erd\H{o}s proved in 1963  that if $h_{k}(G) - \ell_{k}(G) \geq k^{2} + 2k - 4$, then  $G$ contains 
  $k$ vertex-disjoint cycles. For each $k\geq 2$, they also showed an infinite sequence of graphs $G_k(n)$ with $h_{k}(G_k(n)) - \ell_{k}(G_k(n)) = 2k-1$
such that $G_k(n)$ does not have $k$ disjoint cycles.
Recently, the authors proved that, for $k \geq 2$, a bound of $3k$ is sufficient to guarantee the existence of $k$ disjoint cycles and presented for every $k$ a graph
$G_0(k)$ with  $h_{k}(G_0(k)) - \ell_{k}(G_0(k))=3k-1$ and no $k$ disjoint cycles.
 The goal of this paper is to refine and sharpen this result: We show that the Dirac--Erd\H{o}s construction is optimal in the 
sense that for every $k \geq 2$, 
there are only finitely many graphs $G$ with $h_{k}(G) - \ell_{k}(G) \geq 2k$ but no $k$ disjoint cycles. In particular,
every graph $G$ with $|V(G)| \geq 19k$ and $h_{k}(G) - \ell_{k}(G) \geq 2k$ contains $k$ disjoint cycles.
\end{abstract}

\maketitle

{\small{Mathematics Subject Classification: 05C35, 05C70, 05C10.}}{\small \par}

{\small{Keywords: Disjoint Cycles, Minimum Degree, Disjoint Triangles.\\}}{\small \par}

 \section{Introduction}\label{sec:intro}
For a graph $G$, let $|G| = |V(G)|$, $\| G \| = |E(G)|$, and $\delta(G)$ be the minimum degree of a vertex in $G$. 
For a positive integer $k$, define $H_{k}(G)$ to be the subset of vertices with degree at least $2k$ and $L_{k}(G)$ to be the subset of vertices of degree at most $2k - 2$.
Two graphs are \emph{disjoint} if they have no common vertices.

Every graph with minimum degree at least $2$ contains a cycle.  The following seminal result of Corr\'{a}di and Hajnal \cite{C-H} generalizes this fact.

\begin{thm}\cite{C-H}\label{thm:C-H}
Let $G$ be a graph and $k$ a positive integer.  If $|G| \geq 3k$ and $\delta(G) \geq 2k$, then $G$ contains $k$ disjoint cycles.
\end{thm}

Both conditions in Theorem~\ref{thm:C-H} are sharp.  The condition $|G| \geq 3k$ is necessary as every cycle contains at least $3$ vertices.  
Further, there are infinitely many graphs that satisfy $|G| \geq 3k$ and $\delta(G) = 2k - 1$, but contain at most $k-1$ disjoint cycles. 
For example, for any $n\ge 3k$, let   $G_n=K_{n}-E(K_{n-2k+1})$ where $K_{n-2k+1}\subseteq K_n$.
 
 The Corr\'{a}di-Hajnal Theorem inspired several results related to the existence of disjoint cycles in a graph
 (e.g.~\cite{D-E, Di, HSz, Enomoto, Wang, KK-Ore, CFKS,  KK-refCH, K-K-Y, KKMY}).  
This paper focuses on the following theorem of Dirac and Erd\H{o}s~\cite{D-E}, one of the first attempts to generalize Theorem~\ref{thm:C-H}.

\begin{thm}\cite{D-E}\label{thm:D-E}
Let $k \geq 3$  be an integer and $G$ be a graph  with $|H_{k}(G)|  - |L_{k}(G)|  \geq k^{2} + 2k - 4$.  Then $G$ contains $k$ disjoint cycles.
\end{thm}

Dirac and Erd\H{o}s suggested that the bound $k^2+2k-4$ is not best possible and also 
constructed  an infinite sequence of graphs $G_k(n)$ with $h_{k}(G_k(n)) - \ell_{k}(G_k(n)) = 2k-1$
such that $G_k(n)$ does not have $k$ disjoint cycles.
They did not explicitly pose problems, and
it seems that Erd\H{o}s regretted not doing so, as later in \cite{Erdos} he remarked (about \cite{D-E}): 
``This paper was perhaps undeservedly neglected; one reason was that we have few easily quotable theorems there, and do not state any unsolved problems.'' 
Here we consider questions that are implicit in \cite{D-E}. 

For small graphs, the bound of $|H_k(G)|-|L_k(G)|\ge2k$ 
 is not sufficient to guarantee the existence of $k$ disjoint cycles.  Indeed, $K_{3k-1}$ contains at most $k - 1$ disjoint cycles, 
so  for small graphs, a bound of at least $3k$ is necessary.  The authors~\cite{KKM} recently proved that $3k$ is also sufficient.
 
\begin{thm}\cite{KKM}\label{thm:3k}
Let $k \geq 2$ be an integer and $G$ be a graph with $|H_k (G)| - |L_{k}(G)| \geq 3k$.  Then $G$ contains $k$ disjoint cycles.
\end{thm}

There exist graphs $G$ with at least $3k$ vertices and $|H_{k}(G)| - |L_{k}(G)| \geq 2k$ that do not contain $k$ disjoint cycles.  
For example, consider the graph $G_0(k)$ obtained from $K_{3k-1}$ by selecting a subset $S \subseteq V(K_{3k-1})$ with $|S| = k$, removing all edges in $G[ S ]$, 
adding an extra vertex $x$ and the edges from $x$ to each vertex in $S$.  Then $|H_{k}(G_0(k))| - |L_{k}(G_0(k))| = 3k - 2$ and $|G_0(k)| = 3k$,
 but $x$ is not in a triangle, so $G_0(k)$ contains at most $k-1$ disjoint cycles.

In~\cite{KKM}, the authors describe another graph $G_1(k)$, obtained from $G_0(k)$ by  adding $k$ vertices of degree $1$, each adjacent to $x$. 
 The graph $G_1(k)$ still contains only $k-1$ disjoint cycles,  but has $4k$ vertices and $|H_{k}(G_1(k))| - |L_{k}(G_1(k))| = 2k$.  
However, in the special case that $G$ is planar, it is shown in~\cite{KKM} that the bound of $2k$ is sufficient.

\begin{thm}\cite{KKM}\label{thm:planar}
Let $k \geq 2$ be an integer and $G$ be a planar graph.  If
\[ |H_{k}(G)| - |L_{k}(G)| \geq  2k,\]
then $G$ contains $k$ disjoint cycles.
\end{thm}

Further, when $k \geq 3$, a bound of $2k$ is also sufficient for graphs with  no two disjoint triangles.

\begin{thm}\cite{KKM}\label{thm:1tri}
Let $k \geq 3$ be an integer and $G$ be a graph such that $G$ does not contain two disjoint triangles.  If
\[ |H_{k}(G)| - |L_{k}(G)| \geq 2k,\]
then $G$ contains $k$ disjoint cycles.
\end{thm}

In general, the bound of $2k$ is the best we may hope for, as witnessed by $K_{n - 2k + 1, 2k-1}$ for $n \geq 4k$.  
Further, the graph $G_1(k)$ described above shows that a difference of $2k$ is not sufficient when $|G|$ is small.  
In~\cite{KKM}, we were not able
  to determine whether for each $k$ there are only finitely many such examples. In order to attract attention to this
problem and based on known examples, we raised  the following question. 

\begin{question}\cite{KKM}\label{q:main}
Is it true that every graph $G$ with $|G|\geq 4k+1$ and $|H_k(G)| - |L_k(G)| \geq 2k$ has $k$ disjoint cycles?
\end{question}

The goal of this paper is to confirm that indeed for every $k \geq 2$, 
there are only finitely many graphs $G$ with $h_{k}(G) - \ell_{k}(G) \geq 2k$ but no $k$ disjoint cycles. 
We do this by answering Question~\ref{q:main}  for graphs with at least $19k$ vertices.

\begin{thm}\label{thm:2k}
Let $k \geq 2$ be an integer and $G$ be a graph with $|G| \geq 19k$ and
\[ |H_{k} (G)| - |L_{k}(G)| \geq 2k.\]  
Then $G$ contains $k$ disjoint cycles.
\end{thm}

The remainder of this paper is organized as follows.  The next two sections outline notation and previous results that will be used in the proof of Theorem~\ref{thm:2k}. 
 We also introduce Theorem~\ref{thm:2k induct}, which is a more technical version of Theorem~\ref{thm:2k}.  
 Theorem~\ref{thm:2k induct} is proved in Section~\ref{sec:proof}. 
 The proof  builds on the techniques of Dirac and Erd\H{o}s \cite{D-E} and uses Theorem~\ref{thm:3k} as the base case for our induction.

 \section{Notation}\label{sec:notation}

We mostly use the standard notation.  For a graph $G$ and $x\in V(G)$, $N_{G}(x)$ is the set of all vertices adjacent to $x$ in $G$,  and the \emph{degree} of $x$, denoted $d_{G}(x)$, is  $|N_{G}(x)|$.  
When the choice of $G$ is clear, we  simplify the notation to $N(x)$ and $d(x)$, respectively.
The complement of a graph $G$ is denoted by $\overline{G}$.
For an edge $xy\in E(G)$, $G\diagup xy$ denotes the graph obtained from $G$ by contracting $xy$; the new vertex is denoted by 
$v_{xy}$.

For disjoint sets $U, U' \subseteq V(G)$, we write $\| U, U' \|_{G}$ for the number of edges from $U$ to $U'$.  When the choice of $G$ is clear, we will write $\| U, U' \|$ instead.
If $U = \{ u \}$, then we will write $\| u, U' \|$ instead of $\| \{ u \}, U' \|$.
The \emph{join} $G \vee G'$ of two graphs is $G \cup G' \cup \{xx' : x \in V(G)~\mbox{and}~x' \in V(G')\}$. 
 Let $SK_{m}$ denote the graph obtained by subdividing one edge of the complete $m$-vertex graph $K_{m}$. 

Given an integer $k$, we say a vertex in $H_{k}(G)$ is \emph{high}, and set $h_{k}(G)=|H_{k}(G)|$.  A vertex in $L_{k}(G)$ is \emph{low}. Set $\ell_{k}(G)=|L_{k}(G)|$.  A vertex $v$ is in $V^{i}(G)$ if $d_{G}(v) = i$.  Similarly, $v \in V^{\leq i}(G)$ if $d_{G}(v) \leq i$ and $v \in V^{\geq i}(G)$ if $d_{G}(v) \geq i$.
 In these terms, $H_k(G)=V^{\ge2k}(G)$ and $L_k(G)=V^{\le2k-2}(G)$.

We say that $x, y, z \in V(G)$ \emph{form a triangle} $T = xyzx$ in $G$ if $G[\{x,y,z\}]$ is a triangle.  
If $v \in \{x, y, z \}$, then we say $v \in T$.
A \emph{set $\T$ of disjoint triangles} is a set of  subgraphs of $G$ such that each subgraph is a triangle and all the triangles are disjoint. 
For a set $\mathcal S$ of graphs, let $ \bigcup \mathcal S = \bigcup \{V(S):S \in \mathcal S\}$.
For a  graph $G$, let $c(G)$ be the maximum number of disjoint cycles in $G$ and $t(G)$ be the maximum number of disjoint triangles  in $G$.  
When the graph $G$ and integer $k$ are clear from the context, we use $H$ and $L$ for $H_{k}(G)$ and $L_{k}(G)$, respectively.  
The sizes of $H$ and $L$ will be denoted by $h$ and $\ell$, respectively.

 \section{Preliminaries}\label{sec:pre}

As shown in \cite{K-K-Y}, if a graph $G$ with $|G| \geq 3k$ and $\delta(G) \geq 2k - 1$ does not contain a large independent set, then with two exceptions, 
$G$ contains $k$ disjoint cycles:

\begin{thm}\cite{K-K-Y}\label{thm:K-K-Y}
Let $k \geq 2$.  Let $G$ be a graph with $|G| \geq 3k$ and $\delta(G) \geq 2k -1$ such that $G$ does not contain $k$ disjoint cycles.  Then
\begin{enumerate}
\item $G$ contains an independent set of size at least $|G| - 2k + 1$, or
\item $k$ is odd and $G = 2K_{k} \vee \overline{K_{k}}$, or
\item $k =2$ and $G$ is a wheel.
\end{enumerate}
\end{thm}

The theorem gives the following corollary.

\begin{cor}\label{cor:K-K-Y}
Let $k \geq 2$ be an integer and $G$ be a graph with $|G| \geq 3k$.  If $h  \geq 2k$ and $\delta(G) \geq 2k - 1$ (i.e. $L = \emptyset$),
 then $G$ contains $k$ disjoint cycles.
\end{cor}

This corollary, along with the following theorem from \cite{KKM} will be used in the proof.

\begin{thm}\cite{KKM}
\label{thm:2k+t} Let $k\geq2$ be an integer and $G$ be a graph such that $|G| \geq 3k$.  If 
\[
h - \ell \geq 2k + t(G),
\]
then $G$ contains $k$ disjoint cycles. 
\end{thm}

 We prove the following technical statement that implies Theorem~\ref{thm:2k}, but is more amenable to induction.

\begin{thm} 
\label{thm:2k induct} Suppose $i,k\in\mathbb Z$, $k\geq i$ and $k\geq 2$.  Let $\const  = 16$ be a constant.  If $G$ is a graph with $|G| \geq \const  k + 3i$ and $h \geq \ell + 3k - i$,
then $c(G)\geq k$.
\end{thm}

Theorem~\ref{thm:2k} is the special case of Theorem~\ref{thm:2k induct} for $i = k$. 
The heart of this paper will be a proof of Theorem~\ref{thm:2k induct}.  In the remainder of this section we organize the induction and 
 establish some preliminary results.

We argue by induction on $i$. The base case $i \leq 0$ follows from Theorem~\ref{thm:3k}.  Now suppose $i\ge 1$. The equations $|G| \geq h + \ell$ and $h - \ell \geq 2k$  give
\begin{equation}
\label{eq:less than half}
\ell \leq \frac{|G|}{2} - k.
\end{equation}

The \emph{2-core} of a graph $G$ is the largest subgraph  $G'\subseteq G$ with $\delta(G') \geq 2$.  It can be obtained from $G$ by iterative deletion of vertices of degree at most $1$. The following Lemma was proved in  \cite{KKM}.

\begin{lem}\cite{KKM}\label{lemma:2-core}
Suppose the $2$-core of $G$ contains at least $6$ vertices
and  is not isomorphic to $SK_{5}$. If $h_{2}(G) - \ell_{2}(G) \geq 4$ then $c(G)\geq2$.
\end{lem}

Now, we prove a result regarding minimal counterexamples to Theorem~\ref{thm:2k induct}.  Call a triangle $T$ \emph{good} if $T \cap L_{k}(G) \neq \emptyset$.

\begin{lem} 
\label{lem:minimal}Suppose $i,k\in\mathbb Z$, $k\geq i$ and $k\geq 2$.   Let $\const = 16$.  If a graph $G$ satisfies all of:
\begin{enumerate}[	(a)]
\item $|G| \geq \const k + 3i$,
\item $h \geq \ell + 3k - i$,
\item $c(G) < k$, and
\item subject to (a\textendash c), $\sigma := ( k, i, |G|+\left\Vert G\right\Vert )$ is lexicographically minimum,
\end{enumerate}
then all of the following hold:
\begin{enumerate}[(i)]
\item $G$ has no isolated vertices; \label{minimal 1}
\item $k\geq3$; \label{minimal 2}
\item $L(G)\cup V^{\geq 2k+1}(G)$ is independent; \label{minimal 3}
\item if $x\in L(G)$, $d(x)\geq2$, and $xy\in E$, then $xy$ is in a triangle; and \label{minimal 4}
\item if $\mathcal{T}$ is a set of disjoint good triangles in $G$ and $X:=\bigcup\T$, then $\left\Vert v ,X\right\Vert \geq2|\mathcal{T}|+1$
for at least two vertices $v\in V\smallsetminus X$. \label{minimal 5}
\end{enumerate}
\end{lem}

\begin{proof}
Assume (a--d) hold. Using Theorem~\ref{thm:3k}, (a--c) imply $i \geq 1$; so the minimum in (d) is well defined.
If (\ref{minimal 1}) fails, then let $v$ be an isolated vertex in $G$.  Now $G': = G-v$ and $i': = i -1$ satisfy conditions (a\textendash c), contradicting (d). 
 Hence,~(\ref{minimal 1}) holds.

For (\ref{minimal 2}), suppose $k=2$.  Then $t(G)\le c(G) \leq 1$.  If $i =1$ then  $h - \ell \geq 3k - i \geq 2k + t(G)$, so $c(G) \geq 2$ by Theorem~\ref{thm:2k+t}. 
 Thus $i = 2$ and $h - \ell = 4$.  Using~\eqref{eq:less than half} and~(\ref{minimal 1}), 
{\allowdisplaybreaks
\begin{align}
\nonumber \| G \| &\geq \frac{1}{2} ( \ell + 3( |G| - \ell ) + h ) = \frac{1}{2} ( 3|G| + h - 2\ell  ) \\
\nonumber &= \frac{1}{2} ( 3|G| - \ell + 4  ) \geq \frac{1}{2} \left( 3|G| - \left( \frac{|G|}{2} - 2 \right) + 4  \right) \\
\label{j20} &= |G| + \frac{ |G| }{4} + 3
\geq |G| + \frac{\const}{2} + \frac{3i}{4} + 3 = |G| + \frac{\const}{2} + \frac{9}{2}. 
\end{align}
}If $G'$ is the $2$-core of $G$, then $\| G' \| - |G'| \ge \| G \| - |G|$.  Since $\const > 1$,~\eqref{j20} yields $\|G' \| > |G'| + 5$; so $|G'| > 5$ and $G' \not\cong SK_{5}$.  
By Lemma~\ref{lemma:2-core}, $c(G)\ge 2$, contradicting (c).

For (\ref{minimal 3}), suppose $e \in E(G[L\cup V^{\geq 2k+1}(G)])$, and set $G':=G - e$.  Since $G'$ is a spanning subgraph of $G$, it satisfies (a) and (c).
 Moreover, by the definition of $G'$, $h_k(G')=h$ and $ \ell_k(G')=\ell$, so (b) holds for $G'$, which means
 (d) fails for $G$. 

If (\ref{minimal 4}) fails, then let $G' = G \diagup xy$ and $i' = i - 1$.  Since $d_{G'} (v_{xy}) \geq d(y)$ and the degrees of all other vertices in $G'$ are unchanged,
 $G'$ and $i'$ satisfy (a\textendash c), contradicting (d).

Finally, suppose (\ref{minimal 5}) fails, and let $u\in V\smallsetminus X$ with
$\left\Vert u,X\right\Vert $ maximum. Then $\left\Vert v,X\right\Vert \leq2|\mathcal{T}|$
for all $v\in V\smallsetminus(X+u)$. Set $G'=G-X$,
 $k'=k-|\mathcal{T}|$, and  $i' = i - |\mathcal{T}| \leq k'$. 
  Then $H\cap V(G')-u\subseteq H_{k'}(G')$ and $L_{k'}(G')-u\subseteq L\cap V(G')$.
Since $\const \geq 3$, we have  $|G'| \geq \const k' + 3i$; so $G'$ satisfies (a). Let $\beta_{1}=1$ if $u\in H\smallsetminus H_{k'}(G')$;
else $\beta_{1}=0$. Let $\beta_{2}=1$ if $u\in L_{k'}(G)\smallsetminus L$;
else $\beta_{2}=0$. Then $\beta_{1}+\beta_{2}\leq|\mathcal{T}|$ and so
{\allowdisplaybreaks
\begin{align*}
h_{k'}(G') &\geq h - 2|\mathcal{T}| - \beta_{1} 
\geq \ell + 3k - i - 2| \mathcal{T} | - \beta_{1}\geq ( \ell_{k'}(G')+ |\mathcal{T} |-\beta_2)+3k-i
-2| \mathcal{T} | - \beta_{1}\\
&=\ell_{k'}(G') - |\mathcal{T} |  + 3(k'+|\mathcal{T} |) - (i'+  |\mathcal{T}|) - \beta_{1}-\beta_{2} 
\geq \ell_{k'}(G') + 3k' - i'.
\end{align*}}
\noindent
This means $G'$ satisfies (b).
As $c(G')+|\mathcal{T}|\leq c(G)<k$, $c(G')<k'$. Thus $G'$ satisfies (c).  If $k' \geq 2$, then this contradicts the choice of $k$ in (d), so (\ref{minimal 5}) holds.

  Otherwise,  $| \mathcal{T} | = k - 1$ and so $|X | = 3k -3$. Since each triangle in $\T$ has a low vertex,
$|L\cap X|\geq |\T|$,   and  by (iii), $d_G(x)\leq 2k$ for each $x\in X$. 
 Thus
\begin{equation}\label{10}
\| X,V(G')\|<2k|X|<6k^2.
\end{equation}
By (b),   $|H\cap V(G')|- |L\cap V(G')|\ge 3k -i-|H\cap X|+ |L\cap X|\ge 2k-i$. 
So,
$$\sum_{v\in V(G')\cap (H\cup L)}d_G(v)\geq 2k|H\cap V(G')|\geq 2k\frac{|V(G')\cap(H\cup L)|+(2k-i)}{2}.
$$
By this and~\eqref{10}, we get
\begin{equation}\label{11}
2\|G'\|=\sum_{v\in V(G')}d_G(v)-\|X,V(G')\|\ge k(|G'|+2k-i)-\|X,V(G')\|\ge k(|G'|-4k-i). 
\end{equation}
By (c), $c(G)\le k-1$, so $G'$ has no cycle. Thus by~\eqref{11},
\[2|G'|>2\|G'\|\ge k(|G'|-4k-i).\] 
By (a), $|G'|\ge|G|-3k\ge(\alpha-3)k+3i=13k+3i$. Solving yields 
\begin{align}\notag
k(4k+i)&>(k-2)|G'|\ge (k-2)(13k+3i)\\
26k&>9k^2+i(2k-6).\notag 
\end{align}
As $i\ge 0$, and $k\ge3$ by \eqref{minimal 2}, this is a contradiction.
\end{proof}

 \section{Proof of Theorem~\ref{thm:2k induct}}\label{sec:proof}

Fix $k$, $i$, and $G=(V,E)$ satisfying the hypotheses of Lemma~\ref{lem:minimal}. First choose
a set $\mathcal{S}$ of disjoint good triangles with $s:=|\mathcal{S}|$ maximum,
and put $S=\bigcup\mathcal{S}$.
 Next choose a set $\mathcal{S}'$ of disjoint triangles,
each contained in $V^{\leq2k}(G)\smallsetminus S$, with $s':=|\mathcal{S}'|$
maximum, and put $S'=\bigcup\mathcal{S}'$. 
Say $\mathcal{S}=\{T_{1},\dots,T_{s}\}$ and $\mathcal{S}'=\{T_{s+1},\dots,T_{s+s'}\}$. 

Let $\mathcal{H}$ be the directed graph defined on vertex set $\mathcal{S}$ by
$CD\in E(\mathcal{H})$ if and only if there is $v\in C$ with $\left\Vert v,D\right\Vert =3$. 
 Here we allow graphs with no vertices. 
  A vertex $C'$ is \emph{reachable} from a vertex $C$ if 
$\mathcal{H}$ contains a directed $CC'$-path.

\begin{fact}
\label{2}If $x\in L\smallsetminus S$ and $d(x)\geq2$ then $N(x)\subseteq S$. 
\end{fact}

\begin{proof}
Suppose $y\in N(x)\smallsetminus S$. As $x$ is low, $x \notin S'$. By
Lemma~\ref{lem:minimal}(\ref{minimal 4}), $xy$ is in a triangle $xyzx$. As\emph{ $\mathcal{S}$ }is maximal, $z\in S$,
so $z\in C$ for some $C\in\mathcal{S}$. Let 
\[ \mathcal{S}_{0}=\{C'\in\mathcal{S}:C \text{~is reachable from \ensuremath{C'} in}\,\mathcal{H}\}.\]
By Lemma~\ref{lem:minimal}(\ref{minimal 5}), there is $w\in(V\smallsetminus\bigcup\mathcal{S}_{0})-y$ with $\left\Vert w,\bigcup\mathcal{S}_{0}\right\Vert \geq2|\mathcal{S}_{0}|+1$.
Then $\left\Vert w,D\right\Vert =3$ for some $D\in\mathcal{S}_{0}$. By Lemma~\ref{lem:minimal}(\ref{minimal 3}), $w\ne x$.
Further, $w\notin S$ as otherwise the triangle in $\mathcal{S}$ containing $w$ is in $\mathcal{S}_{0}$, contradicting that $w \notin \bigcup \mathcal{S}_{0}$.

Let $D = C_{1}, \ldots, C_{j} = C$ be a $D,C$-path in $\mathcal{H}$, and for $i \in [j - 1]$ let $x_{i} \in C_{i}$ with $\| x_{i}, C_{i+1} \| = 3$.  If $C_{1}' = C_{1} - x_{1} + w$, $C_{j}' = C_{j} - z + x_{j-1}$ and $C_{i}' = C_{i} - x_{i} + x_{i-1}$ for $i \in \{ 2, \ldots, j-1\}$, then $\left( \mathcal{S} \smallsetminus \bigcup_{i = 1}^{j} C_{i} \right) \cup \bigcup_{i =1}^{j} C_{i}' \cup \{xyzx\}$ is a set of $s+1$ disjoint good triangles.  
 This contradicts the maximality of $\mathcal{S}$.
\end{proof}

\begin{fact}\label{fact:2 leaves}
Each $v \in V$ is adjacent to at most $2$ leaves.  Moreover, if $v$ is adjacent to $2$ leaves, then $v \in V^{2k}$.
\end{fact}

\begin{proof}
Let $v$ be adjacent to a leaf.  By Lemma~\ref{lem:minimal}(\ref{minimal 3}), $v \in V^{2k-1} \cup V^{2k}$.  Let $X$ be the set of leaves adjacent to $v$, and put $G'=G-X$. 
 Let $i' = i - (|X| - 1 - |\{ v \} \cap V^{2k}|)$. Observe
{\allowdisplaybreaks
\begin{align*}
h_{k} ( G' ) - \ell_{k} (G' ) &\geq (h - |\{ v \} \cap V^{2k}| ) -  (\ell + 1 - |X| )  \\
&= h - \ell - |\{ v \} \cap V^{2k}| + |X| - 1 \\
&\geq 3k  - i - |\{ v \} \cap V^{2k}| + |X| - 1 \\
&\geq 3k - i',
\end{align*}}
so (b) holds for $G', k$ and $i'$.  Now, $|G'| \geq \const k + 3i - |X| = \const k + 3i' + 2|X| - 3( 1 + |\{v \} \cap V^{2k}|)$.
If $|X| \geq 3$, then $2|X| - 3( 1 + |\{ v \} \cap V^{2k}|) \geq 0$, so $|G'| \geq \const k + 3i'$ and (a) holds.  As $i'$ is at most $i$ and $G' \subset G$, (d) does not hold for $G,k$, and $i$, a contradiction.
 Similarly, if $v \in V^{2k-1}$ and $|X| = 2$, then $|G'| \geq \const k + 3i'$, so (a) still holds and $G', k$ and $i'$.
 Thus this also contradicts (d) for $G$.
\end{proof}

Let $G_{1} = G - V^{1}$. 
 Let $H^1 = V^{\geq2k}(G_{1})$,
$R^1=V^{2k-1}(G_{1})$, $L^1 = L_{k}(G_{1}) \cap L$, and $M = L_{k}(G_{1})\smallsetminus L^1$.
Then $G_{1} = G[H^1\cup R^1 \cup M \cup L^1]$ and $V^{\geq2k-1}(G) = H^1 \cup R^1 \cup M$.
Since deleting a leaf does not decrease the difference  $h- \ell$, 
\begin{equation}
\label{eq:G1}
h_{k}(G_{1}) - \ell_{k} (G_{1})  \geq 3k - i.
\end{equation}  

\begin{fact}\label{fact:new low in triangle}
If $x \in M$, then $x$ is in a triangle $xyzx$ in $G$ with $d(x), d(y), d(z) \leq 2k$.
\end{fact}

\begin{proof}
Suppose $x \in M$. By Fact~\ref{fact:2 leaves}, either (i) $x \in V^{2k-1}$ and is adjacent to one leaf or (ii) $x \in V^{2k}$ and is adjacent to two leaves. Thus $d(x)\le 2k$. We first claim:
\begin{equation}
\label{CF3.3}
\mbox{\em $x$ has a neighbor $y$ such that $2\leq d(y) \leq 2k$.}
\end{equation}
Suppose not.  Let $X$ be the set consisting of $x$ and the leaves adjacent to $x$.  For each vertex $v \not\in X$, $d_{G-X}(v) \geq d(v) - 1$, with equality  if $v \in N(x)$.  Moreover, if $v \in N(x)$, then $d_{G-X} (v) \geq 2k$.  Therefore, $h_{k}( G - X ) = h - | \{ x \} \cap V^{2k}|$ and $\ell_{k}( G - X ) = \ell - ( |X| - 1 )$.  So 
\[ h_{k}(G - X ) - \ell_{k} (G - X) = h - \ell + 1 \geq 3k - ( i - 1 ) \]
and 
$|G - X| \geq |G| - 3 \geq \const k + 3( i - 1 )$,
 contradicting the minimality of $i$.  So \eqref{CF3.3} holds.

Now, suppose $xy$ is not in a triangle.  Let $G'$ be formed from $G$ by removing the leaves adjacent to $x$ and contracting $xy$.  By Fact~\ref{fact:2 leaves}, $|G'| \geq |G| - 3$.  Since $d(x) \geq 2k - 1$ and $x$ does not share neighbors with $y$, $d_{G'}(v_{xy}) \geq d(y)$.  Similarly, $d_{G'}(v) = d(v)$ for all $v \in V(G') - v_{xy}$.  Now, $h_{k}(G') - \ell_{k}(G') = h - \ell + 1 \geq 3k - (i-1)$,  contradicting the choice of $i$.  

Let $xyzx$ be a triangle containing $xy$.  If $d(z) \leq 2k$, we are done.
Otherwise, let $G''$ be the graph obtained from $G$ by removing the leaves adjacent to $x$ and deleting the vertices $x, y$, and $z$.  Observe $|G''| \geq |G| - 5 \geq \const (k-1) + 3(i - 1)$.  If there exists a vertex $u \in H \setminus H_{k-1}(G'')$, then $N(u) \supseteq \{ x, y, z \}$,  and $d(u) \leq 2k$, since $d(z) \geq 2k+ 1$.  In this case $xyux$ is the desired triangle.  Similarly, if $v \in L_{k-1}(G'') \setminus L$, then $xyvx$ is the desired triangle.  Thus  $h - h_{k-1}(G'') \leq 2 + |\{ x \} \cap V^{2k}|$ and 
$\ell - \ell_{k-1}(G'') \geq  1 + |\{ x \} \cap V^{2k}|$.  Now,
\[ h_{k-1}( G'') - \ell_{k-1}(G'') \geq h - \ell - 1 \geq 3k - i - 1 = 3(k-1) - (i - 2).\]
 By the minimality of $G$, $c(G'') \geq k-1$. Hence $c(G) \geq k$, a contradiction.  We conclude that $xyzx$ is a triangle with $d(x), d(y), d(z) \leq 2k$.
\end{proof}

\begin{fact}\label{fact:s plus s prime}
$s + s' \geq 1$.
\end{fact}

\begin{proof}
Suppose $s + s' = 0$.  In this case, Fact~\ref{fact:new low in triangle} implies $M = \emptyset$: indeed, if $v \in M$, there exists a triangle $vuwv$ with $d(v), d(u), d(w) \leq 2k$, contradicting the choice of $\mathcal{S}'$.  By Fact~\ref{2} and since $\mathcal{S} = \emptyset$, all vertices in $L$ have degree at most $1$.  By Lemma~\ref{lem:minimal}(\ref{minimal 1}), all vertices in $L$ are leaves in $G$ and $L^1 = \emptyset$.

Now, for every $x \in H - H_{k}(G_{1})$, there is a leaf $y \in L - L_{k}(G_{1})$ such that $xy \in E(G)$. Hence,
\[ h_{k}(G_{1}) \geq h_{k}(G_{1}) - \ell_{k}(G_{1}) \geq h - \ell \geq 2k. \] 
By~\eqref{eq:less than half} and since $\const \geq 4$, $|G_{1}| \geq |G| - \ell \geq |G|/2 + k \geq \const k / 2 + k\geq 3k$.  Finally, $L_{k}(G_{1}) = L^1 \cup M = \emptyset$, so Corollary~\ref{cor:K-K-Y} implies $G_{1}$ (and also $G$) contains $k$ disjoint cycles.
\end{proof}

Let $G_{2} = G\smallsetminus(L\smallsetminus S)$. 
 By~\eqref{eq:less than half},
\begin{equation}
|G_{2}| \geq \frac{\const +2}{2}k +\frac{3i}{2}. \label{eq:B}
\end{equation}

\begin{proof}[Proof of Theorem~\ref{thm:2k induct}]
 Put $s^{*}=\max\{1,s\}$. Let $\mathcal{S}^{*}=\{T_{1},\dots,T_{s^{*}}\}$; by
Fact~\ref{fact:s plus s prime}, $T_{s^{*}}$ exists. Put $S^{*}=\bigcup\mathcal{S}^{*}$.
Let $W=V(G_{2})\smallsetminus S^{*}$, $F=G[W]$ and $k'=k-s^{*}$. It suffices to
prove $c(F)\geq k'$. \\

\noindent \emph{Case 1:} $s^{*} = k-1$.  Since $k \geq 3$, $s^{*} \geq 2$.  Thus, $s = s^{*} = k-1$.  By Fact~\ref{fact:2 leaves}, all vertices in $M$ have degree $2k-2$.  Let $M' = M \cap W$ and $H' = H(G_{2}) \cap W$.  Fact~\ref{2} implies that if $v \in W$, then $d_{G_{1}}(v) = d_{G_{2}}(v)$. Thus
\[ H' = H^1 \cap W \text{ and } L(G_{1})\cap W = L(G_{2}) \cap W.\] 
Hence, by~\eqref{eq:G1},
{\allowdisplaybreaks
\begin{align}
2k \leq h(G_{1}) - \ell( G_{1} ) &\leq ( |H(G_{1}) \cap S| + |H'| ) - ( |L(G_{1}) \cap S| + |M \cap W| + |L^1 \smallsetminus S| ) \nonumber \\
&= (|H(G_{1}) \cap S| - |L(G_{1}) \cap S| ) + |H'| - |M'| - |L^1 \smallsetminus S| \label{eq:H minus M} \\
&\leq (k-1) + |H'| - |M'|. \nonumber
\end{align}
}
Here, the last inequality holds because $S$ contains $s = k-1$
low vertices and at most $2s = 2k - 2$ high vertices.  Equation~\eqref{eq:H minus M} implies $|H'| - |M'| \geq k + 1$.  Further, if $W$ does not contain a cycle, then
\begin{align}
\| W, S  \|_{G_{2}} &\geq \sum_{v \in W} d_{G_{2}} (v) - 2( |W| - 1 ) \nonumber \\
&\geq ( (2k-1)|W| + |H'| - |M'| ) - 2(|W| - 1) \nonumber \\
&\geq ( (2k-1)|W| + k + 1 ) - 2(|W| - 1) \label{eq:edges upper} \\
&\geq (2k - 3 )|W| + k + 3. \nonumber 
\end{align}
On the other hand, using Lemma~\ref{lem:minimal}\eqref{minimal 3}, 
\begin{equation}
\label{eq:edges lower}
\| W, S \|_{G_{2}} \leq \sum_{w \in S} (d_{G_{2}} (w) - 2) \leq ( k-1 )( 6k - 8).
\end{equation}
Therefore, combining~\eqref{eq:edges upper} and~\eqref{eq:edges lower}, $|W| \leq 3( k - 1) - \frac{4}{2k- 3}$. Since $|S| = 3(k - 1)$ and $|G_{2}| = |S| + |W|$, this contradicts \eqref{eq:B} when $\const \geq 10$. \\

\noindent \emph{Case 2: } $s^{*} \leq k - 2$.  Consider a vertex $v$ in $V^{\leq 2k'-2}(F)$.  Since every vertex in $F$ has degree at least $2k - 2$ in $G_{2}$, $v$ must be adjacent to at least $2s^{*}$ vertices in $S^{*}$.  Further, every vertex in $S^{*}$ is adjacent to at most $2k - 2$ vertices outside of $S^{*}$.  Therefore,  
\begin{equation}
\label{eq:V2k-2edge}
2s^{*} |V^{\leq 2k'-2}(F) | \leq \|  V^{\leq 2k'-2} (F), S^{*} \| \leq 3s^{*} (2k - 2),
\end{equation}
and so
\begin{equation}
\label{eq:V2k-2}
|V^{ \leq 2k'-2}(F)| \leq 3k - 3.
\end{equation}
Similarly, if $u \in V^{2k'-1}(F)$, then $u$ is adjacent to at least $2s^{*} - 1$ vertices in $S^{*}$.  Moreover, 
there are at most $3s^{*}(2k - 2) - \| V^{\leq 2k'-2}(F), S^{*} \|$ 
edges from $V^{2k-1}(F)$ to $S^{*}$.  So,
\[(2s^{*} - 1) | V^{2k'-1}(F) | \leq \| V^{2k'-1} (F), S^{*} \| \leq 3s^{*} (2k - 2) - \| V^{\leq 2k'-2} (F), S^{*} \| ,\]
and, combining with~\eqref{eq:V2k-2edge} gives,
\begin{align}
|V^{2k'-1}| &\leq \frac{2s^{*}(3k - 3)}{2s^{*} - 1} - \frac{2s^{*} |V^{\leq 2k'-2}(F)|}{2s^{*} - 1} \nonumber \\
&= 3k - 3 + \frac{3k-3}{2s^{*} - 1} - \frac{2s^{*} |V^{\leq 2k'-2}(F)|}{2s^{*} - 1}. \label{eq:V2k-1}
\end{align}
Using~\eqref{eq:V2k-2} and~\eqref{eq:V2k-1}, we see that
\begin{align*}
h_{k'}(F) - \ell_{k'}(F) &= |W| - 2|V^{\leq 2k'-2}(F)| - |V^{2k'-1}(F)| \\
&\geq |W| - 2|V^{\leq 2k'-2}(F)| - \left(3k - 3 + \frac{3k-3}{2s^{*} - 1} - \frac{2s^{*} |V^{\leq 2k'-2}(F)|}{2s^{*} - 1}\right) \\
&= |W| - \frac{ (2s^{*} - 2 )|V^{\leq 2k'-2}(F)|}{2s^{*} - 1} - 3k + 3 - \frac{3k-3}{2s^{*} - 1}  \\
&\geq |W| - \frac{ (2s^{*} - 2 )(3k - 3)}{2s^{*} - 1} - 3k + 3 - \frac{3k-3}{2s^{*} - 1}  \\
&= |W| + \left( -( 3k - 3) + \frac{ 3k - 3}{2s^{*} - 1} \right) - 3k + 3 - \frac{3k-3}{2s^{*} - 1}  \\
&= |W| - 6k + 6 \\
&\geq \left( \frac{\const +2}{2}k +\frac{3i}{2} - 3s^{*} \right)- 6k + 6 \\
&\geq \frac{\const +2}{2}k +\frac{3i}{2} - 9k + 6 + 3k'. \\
\end{align*}
When $\const  \geq 16$, this is at least $3k'$ and $F$ contains $k'$ disjoint cycles by Theorem~\ref{thm:3k}.
\end{proof}

\section*{Acknowledgement}
The authors thank Jaehoon Kim for helpful comments.

\bibliographystyle{abbrv}

\end{document}